\documentclass[12pt, reqno]{amsart}
\usepackage{fullpage}
\usepackage{graphicx}
\usepackage{amsmath, amssymb}
\usepackage{physics}
\usepackage{bbm,bm}
\usepackage{hyperref}
\usepackage{enumerate}
\usepackage{xcolor}

\newtheorem{theorem}{Theorem}
\newtheorem{lemma}{Lemma}

\theoremstyle{definition}
\newtheorem{example}{Example}
\newtheorem{subexample}{Example}[example]

\numberwithin{equation}{section}

\newcommand{\R}{\mathbb{R}}
\newcommand*\diff{\mathop{}\!\mathrm{d}}
\def\nl{\\ \noalign{\medskip}}
\def\sign{{\rm sign}}
\def\bfF{\mathbf{F}}
\def\calE{\mathcal{E}}
\def\ds{\partial_s}
\def\bfz{\mathbf{z}}
\def\bfu{\mathbf{u}}
\def\bfdz{\bm{\delta}\mathbf{z}}
\def\dphi{\delta\varphi}
\def\bfom{\bm{\omega}}
\def\da{\partial_\alpha}
\def\dt{\partial_t}

\definecolor{heavyred}{cmyk}{0,1,1,0.25}
\hypersetup{
pdftitle=Title,
pdfpagemode=UseOutlines,
citebordercolor=0 0 1,
colorlinks=true,
allcolors=heavyred,
breaklinks=true,
pdfauthor={Angeles, G.M., Barber, J., Schmeiser, C.},
pdfpagetransition=Dissolve,
}

\title{Pitchfork bifurcation and traveling waves \\for a planar ensemble of rigid filaments \\with repulsive interaction} 

\author{Gervy Marie Angeles$^*$}
\address{Institute of Mathematics, University of the Philippines Diliman, Quezon City 1101, Philippines}
\email{gmangeles1@up.edu.ph}

\author{Jared Barber}
\address{Department of Mathematical Sciences, Indiana University Indianapolis, Indianapolis, Indiana, USA}
\email{jarobarb@iu.edu}

\author{Christian Schmeiser}
\address{Faculty of Mathematics, University of Vienna, Oskar Morgenstern-Platz 1, 1090 Wien, Austria}
\email{Christian.Schmeiser@univie.ac.at}

\date{\today}

\thanks{$^*$Corresponding author}
\thanks{Abbreviated title: Bifurcation and waves in repulsive filaments}

\begin{document}
\begin{abstract}
The so-called \emph{Filament Based Lamellipodium Model} is a complex modeling framework for a very heterogeneous chemo-mechanical system of cell biology.
It contains a model for Coulomb repulsion between filaments, whose effect on the stability of the system has been unclear. In this work, a strongly simplified version
of the model is considered, showing a destabilizing effect of the repulsion. This instability results in a pitchfork bifurcation with an additional rotational symmetry, leading to a two-dimensional bifurcating manifold of traveling wave solutions.
The simplified model is derived, its linearization around the trivial steady state is analyzed, and a formal bifurcation analysis is carried out. It is shown that the pitchfork 
bifurcation maybe super- or sub-critical. Time dependent numerical simulations illustrate these results and provide additional, more global information on the emergence of periodic and chaotic dynamics by secondary bifurcations.
\end{abstract}

\maketitle

\section{Introduction}
This work is concerned with a theoretical investigation of a strongly simplified version of the Filament Based Lamellipodium Model (FBLM) \cite{MOSS}. The FBLM is a 
two-dimensional anisotropic continuum model for the dynamics of the actin filament network in the lamellipodium of biological cells spread out on flat substrates.
The filaments (polymers of the protein actin) interact with each other, and with a large number of other intra- and extra-cellular components provide regulation and 
mechanical stability of the network. Consequentially, the full FBLM is very complex and not accessible to serious studies of its analytical properties. However, various versions 
have been used to simulate cellular behaviour \cite{OelzCSSmall08,OelzCS12,MOSS,HirschManhartCS17,SfakianakisPeurichardBrunkCS18}. Analytical results are only available for very much simplified submodels \cite{OelzCS10,ManhartCS17}.

A fundamental question is stability or the related question, in biological terms, of homeostasis. Also destabilizing mechanisms are acceptable, since the model is expected
to describe, among others, mechanisms like the transition from nonmoving to moving states \cite{HirschManhartCS17}. This study is concerned with a caricature of the FBLM,
which cannot be expected to have biological relevance, but which concentrates on a particular mechanism, which turns out to have both stabilizing and destabilizing effects.
We call it \emph{pressure}, and it results from Coulomb repulsion of neighboring actin filaments, which are known to typically carry a slight negative electric charge 
\cite{MOSS,ref48fromMOSS}. 

We consider an ensemble of filaments, described as rigid rods of equal lengths, placed side by side on a flat surface. They slide across the surface against the action
of a \emph{friction} force, in the biological context caused by dynamic adhesions connecting the cell with the substrate. The pressure is modeled by prescribing a potential
energy term depending on a transversal filament density, see the following section. To confine the ensemble against the effect of repulsion, we prescribe \emph{inward forces} on the 
two outer filaments. One effect of the pressure is to avoid concentrations of filaments. However our analysis will show that there is also a destabilizing effect in the longitudinal 
direction. This will be balanced by a \emph{tension} force between the centers of mass of the filaments. This force does not have any biological significance. In the full FBLM cell
membrane tension and myosin contraction have similar effects.

In the following section a mathematical model for the system described above is derived. It is based on a microscopic description for a finite number of filaments. Then a continuum
limit is carried out, where the number of filaments tends to infinity and the typical distance between neighboring filaments tends to zero. The resulting model is a system of PDEs
for the curve of centers of mass of the filaments and for the filament orientations. Both unknowns depend on time and on a Lagrangian variable, the continuum version of a filament 
counter. The system is nonlinear and seems to be parabolic for physically relevant solutions.
The model is rotationally symmetric and possesses a one-parameter family of trivial steady states, where the parameter determines the orientation. 

In Section \ref{sec:lin}, the linearization around the trivial steady state is analyzed. The essential part of the linearized problem is a cross-diffusion system, which is shown to be 
well posed. The null space of the linearized operator is generically one-dimensional due to the rotational symmetry. However, a steady state \emph{bifurcation} occurs for a critical 
value of a dimensionless parameter, measuring the combined strength of the inward forces and the tension relative to the strength of the pressure. For large enough values of this
parameter, the spectral stability of the set of trivial steady states is shown.

These results raise the expectation that the trivial steady states can be destabilized by the effect of the pressure. This is verified in Section \ref{sec:bif} by a formal bifurcation
analysis. Since the problem has an additional flip symmetry, a \emph{pitchfork bifurcation} occurs, which may be super- or sub-critical, depending on the value of an additional
dimensionless parameter. 
Crucially, due to the system's inherent rotational symmetry, we do not just have a standard pitchfork bifurcation; rather, the bifurcating manifold is two-dimensional.
As a consequence of this symmetry, the computations are rather involved and, therefore, aided by using {\tt Mathematica} \cite{Mathematica-reference}. The resulting normal form is a set of two ODEs. One of them governs the amplitude of the bifurcation, whereas the second unknown passively follows the dynamics of the first to account for the rotational invariance. The bifurcating solutions are stationary only in relation to the shape of the filament ensemble, which is deformed compared to the trivial steady state. Since in the deformed states the outer filaments are not parallel, the ensemble moves with constant velocity, pushed by the inward forces.

In the last section, the bifurcation results are illustrated by numerical simulations of the time dependent problem. A finite volume discretization is used for the spatial variable
and a standard ODE solver for the time discretization. Typical solutions for parameter values before and after the bifurcation are computed. In the case of the subcritical bifurcation, the simulation
shows convergence to a deformed state far from the trivial equilibrium. This suggests that the bifurcating branch, which is initially unstable, undergoes a secondary fold bifurcation.
Furthermore, beyond these traveling wave solutions, our numerical simulations reveal a much richer dynamical landscape, including the emergence of both periodic and chaotic dynamics within the network.

\section{Model derivation}

For $N\ge 1$ we consider $N+1$ filaments of length $L>0$ in the plane, parametrized at time $t$ by arclength $s$ (measured from the center of mass):
$$
\bfF_j(s,t) \in \R^2 \,,\qquad 0\le j\le N \,,\quad -\frac{L}{2} \le s \le \frac{L}{2}\,, \quad t\ge 0 \,.
$$
The dynamics of the filaments will be derived by a variational approach, where the interactions of filaments with the environment and with other filaments are described by
energy functionals. 

\subsection*{Friction.} Time discretization is needed to put friction in a variational framework. With a time step $\Delta t>0$ and a friction coefficient
$\mu^F>0$ we consider the quadratic {\it friction energy}
$$
\calE_{N,\Delta t}^F := \frac{\mu^F}{2\Delta t \,N} \sum_{j=0}^N \int_{-L/2}^{L/2} \left| \bfF_j(s,t) - \bfF_j(s,t-\Delta t)\right|^2 ds \,.
$$

\subsection*{Pressure by repulsion.} The filaments are assumed to be ordered in the sense that filament $j-1$ is supposed to lie left of filament $j$, when looking in the positive $s$-direction. This can be described
by the requirement 
$$
\left( \bfF_j - \bfF_{j-1}\right)^\bot \cdot \ds\bfF_j > 0 \,,\qquad j=1,\ldots,N \,,
$$
with $(a,b)^\bot := (-b,a)$. If we expect the filaments to cover a finite region as $N\to\infty$, a reasonable definition of a filament density $\rho_j(s,t)$ is given by
$$
\frac{1}{\rho_j} := N\left( \bfF_j - \bfF_{j-1}\right)^\bot \cdot \ds\bfF_j  \,,\qquad j=1,\ldots,N\,.
$$
The order (i.e. $\rho_j < \infty$) should be preserved in the dynamics by considering the {\it pressure energy}
$$
\calE_N^P := \frac{\mu^P}{N} \sum_{j=1}^N \int_{-L/2}^{L/2} \log (\rho_j L_\text{ref})\, ds \,,
$$
with $\mu^P>0$. This can be motivated physically by Coulomb repulsion, assuming the filaments to be electrically charged (see \cite{MOSS} for details). The reference length $L_\text{ref}$
is only introduced for dimensional consistency at this point. A reasonable choice for its value will be given below.

\subsection*{Boundary forces.} The repulsive forces generated by the pressure are counter-acted by forces exerted on the first and last filaments, which can be described variationally by the {\it boundary energy}
$$
\calE_N^B := \frac{1}{L}\int_{-L/2}^{L/2} (\bfF_0\cdot \mathbf{f}_0 - \bfF_N\cdot \mathbf{f}_1)ds \,,
$$
with the forces per length $\mathbf{f}_0(s,t)$ and $\mathbf{f}_1(s,t)$. 

\subsection*{Tension.} We shall see that the pressure also induces instability in the direction of the filaments. In the Filament Based Lamellipodium Model \cite{MOSS} this can be balanced, e.g., by including 
the effect of membrane tension. Here a stabilization is introduced by assuming an elastic connection between the centers of mass, described by a {\it tension energy} proportional to 
the total length of the polygon consisting of these points:
$$
\calE_N^T := \mu^T \sum_{j=1}^N |\bfF_j(0,t) - \bfF_{j-1}(0,t)| \,,\qquad \mu^T > 0 \,.
$$

\subsection*{Continuum limit.} We pass to the continuum limit $N\to\infty$. The discrete variable $j\in \{0,\ldots,N\}$ is replaced by $\alpha\in [0,1]$ with the correspondence 
$$
\bfF_j(s,t) \approx \bfF(\alpha_j,s,t)  \,,\qquad \alpha_j = \frac{j}{N} \,.
$$
The continuum limits of the four energy contributions are given by
\begin{eqnarray*}
\calE_{\Delta t}^F &=& \frac{\mu^F}{2\Delta t } \int_0^1 \int_{-L/2}^{L/2} \left| \bfF(\alpha,s,t) - \bfF(\alpha,s,t-\Delta t)\right|^2 ds\,d\alpha \,,\\
\calE^P &=& \mu^P \int_0^1 \int_{-L/2}^{L/2} \log (\rho(\alpha,s,t)L_\text{ref}) \, ds\,d\alpha \,,\qquad\mbox{with }    \frac{1}{\rho} = \da\bfF^\bot \cdot \ds\bfF\,,\\
\calE^B &=& \frac{1}{L}\int_{-L/2}^{L/2} \left(\bfF(0,s,t)\cdot \mathbf{f}_0(s,t) - \bfF(1,s,t)\cdot \mathbf{f}_1(s,t)\right)ds \,,\\
\calE^T &=& \mu^T \int_0^1 |\da\bfF(\alpha,0,t)| d\alpha   \,.
\end{eqnarray*}
A continuous time dynamics is obtained by minimizing the sum of the four energy contributions and then passing to the limit $\Delta t\to 0$. However, we add a final modeling step:

\subsection*{Rigid filaments.} We restrict to rigid filaments determined by their centers of mass $\bfz(\alpha,t)$ and direction angles $\varphi(\alpha,t)$:
$$
\bfF(\alpha,s,t) = \bfz(\alpha,t) + s\, \bfom(\varphi(\alpha,t)) \,,\qquad \bfom(\varphi) = (\cos\varphi,\sin\varphi) \,,
$$
 i.e., the minimization is carried out only with respect to $(\bfz,\varphi)$. For rigid filaments the energy contributions become
\begin{eqnarray*}
\calE_{\Delta t}^F &=& \frac{\mu^F L}{2\Delta t } \int_0^1 \left[ \left| \bfz(\alpha,t) - \bfz(\alpha,t-\Delta t)\right|^2 + \frac{L^2}{12} (\varphi(\alpha,t) - \varphi(\alpha,t-\Delta t))^2\right] d\alpha \,,\\
\calE^P &=& \mu^P \int_0^1 \int_{-L/2}^{L/2} \log (\rho(\alpha,s,t)L_\text{ref}) \, ds\,d\alpha \,,\qquad\mbox{with }    \frac{1}{\rho} = \da\bfz^\bot \cdot \bfom(\varphi) - s\,\da\varphi \,,\\
\calE^B &=& \bfz(0,t)\cdot \frac{1}{L}\int_{-L/2}^{L/2} \mathbf{f}_0(s,t) ds  - \bfz(1,t)\cdot \frac{1}{L}\int_{-L/2}^{L/2}\mathbf{f}_1(s,t) ds \,,\\
\calE^T &=& \mu^T \int_0^1 |\da\bfz(\alpha,t)| d\alpha   \,.
\end{eqnarray*}

\subsection*{Nondimensionalization.} A scaling is introduced by
$$
s\to Ls \,,\quad \bfz\to L_\text{ref}\ \bfz \,,\quad (t,\Delta t)\to \frac{\mu^F L_\text{ref}^2}{\mu^P}(t,\Delta t) \,,\quad  \rho\to \frac{\rho}{L_\text{ref}} \,,
$$ $$ \mathbf{f}_i \to f_\text{ref}\ \mathbf{f}_i \,,\quad \calE^X \to \mu^P L \calE^X \,,
$$
for $i=0,1$ and $X = F,P,B,T$, where $f_\text{ref}>0$ is a typical value of $|\mathbf{f}_i|$. It remains to choose a value for $L_\text{ref}$\,. It seems reasonable to balance the
strength of the pressure and the combined strengths of tension and boundary forces:
$$
L_\text{ref} := \frac{\mu^P L}{f_\text{ref} + \mu^T} \,.
$$
The scaled energy contributions are then given by
\begin{eqnarray*}
\calE_{\Delta t}^F &=& \frac{1}{2\Delta t } \int_0^1 \left[ \left| \bfz(\alpha,t) - \bfz(\alpha,t-\Delta t)\right|^2 + \beta^2 (\varphi(\alpha,t) - \varphi(\alpha,t-\Delta t))^2\right] d\alpha \,,\\
\calE^P &=& \int_0^1 \int_{-1/2}^{1/2} \log (\rho(\alpha,s,t)) \, ds\,d\alpha \,,\qquad\mbox{with }    \frac{1}{\rho} = \da\bfz^\bot \cdot \bfom(\varphi) - \sqrt{12}\,\beta s\,\da\varphi \,,\\
\calE^B &=& (1-\gamma) \left( \bfz(0,t)\cdot \overline{\mathbf f}_0(t)  - \bfz(1,t)\cdot \overline{\mathbf f}_1(t) \right)\,,\qquad \mbox{with } \overline{\mathbf f}_i(t) = \int_{-1/2}^{1/2}\mathbf{f}_i(s,t) ds \,,\\
\calE^T &=& \gamma \int_0^1 |\da\bfz(\alpha,t)| d\alpha   \,.
\end{eqnarray*}
with dimensionless parameters
$$
\beta = \frac{L}{\sqrt{12}\, L_\text{ref}} = \frac{f_\text{ref} + \mu^T}{\sqrt{12}\,\mu^P} \,,
\quad \text{and} \quad \gamma = \frac{\mu^T}{f_\text{ref} + \mu^T} \,,
$$
where $\beta>0$ measures the aspect ratio between filament length and expected width of the filament band, and $\gamma\in (0,1)$ measures the relative strength of tension against boundary forces.

\subsection*{The variations.} As mentioned above, the dynamics is determined by choosing $(\bfz(\alpha,t),\varphi(\alpha,t))$ as minimizer of the sum of the four energy contributions in the limit
$\Delta t\to 0$. Therefore we compute the variations of the energy contributions in the direction $(\bfdz(\alpha),\dphi(\alpha))$.

For the friction energy we obtain
\begin{align*} 
\delta\mathcal{E}_{\Delta t}^F = \int_{0}^{1} \left(\frac{\bfz(t) - \bfz(t-\Delta t)}{\Delta t} \cdot\bfdz 
+ \beta^2 \frac{\bfom(\varphi(t)) - \bfom(\varphi(t-\Delta t))}{\Delta t} \cdot\bfom(\varphi(t))^{\perp} \delta \varphi \right)d\alpha \,.
\end{align*}
Taking the limit as $\Delta t \to 0$, we have
\begin{align} \label{varF}
\lim_{\Delta t \to 0} \delta\mathcal{E}_{\Delta t}^F = \int_{0}^{1} \left(\dt\bfz \cdot\bfdz 
+ \beta^2 \dt\varphi \, \delta \varphi \right)d\alpha \,.
\end{align}	
The variation of the pressure energy is given by
$$ 
\delta \mathcal{E}^P = \int_0^1 \left( P_0 \bfom^\perp \cdot \da (\bfdz) - P_0\, \da \bfz \cdot \bfom \,\delta \varphi + \beta^2 P_1 \da (\delta \varphi) \right) d\alpha \,,
$$
where 
$$
P_0 := \int_{-1/2}^{1/2} \rho \, ds \,,\qquad  P_1 := \frac{\sqrt{12}}{\beta}\int_{-1/2}^{1/2} s\rho \, ds\,.
$$
Note that the scaling of $P_1$ is chosen such that it remains finite as $\beta\to 0$ (i.e., for relatively small filament length).
Integration by parts gives
\begin{align} 
\delta \mathcal{E}^P = &- \int_0^1 \left(  \da(P_0 \bfom^\perp)\cdot\bfdz + (P_0\, \da \bfz \cdot \bfom + \beta^2 \da P_1) \delta \varphi \right) d\alpha \nonumber\\
& + \left( P_0 \bfom^\perp\cdot\bfdz + \beta^2 P_1 \delta\varphi\right)\Big|_{\alpha=0}^1 \,.  \label{varP} 
\end{align}
The variation of the tension energy reads
\begin{align}\label{varT}
\delta \mathcal{E}^T = \gamma \int_{0}^{1} \frac{\da \bfz}{\abs{\da\bfz}} \cdot \da (\bfdz) d\alpha 
= -\gamma \int_{0}^{1} \da\left(\frac{\da \bfz}{\abs{\da\bfz}}\right) \cdot \bfdz \,d\alpha	+ \gamma \frac{\da \bfz}{\abs{\da\bfz}} \cdot \bfdz\Big|_{\alpha=0}^1 \,.
\end{align}
For the linear boundary energy, the variation is straightforward:	
\begin{align}\label{varB}
\delta \mathcal{E}^B = (1-\gamma) \left( \bfdz(0)\cdot \overline{\mathbf f}_0  - \bfdz(1)\cdot \overline{\mathbf f}_1 \right)
\end{align}
To act against the pressure, we prescribe the forces to be equally strong and to push inwards, i.e.,
\begin{align}\label{forces}
\overline{\mathbf f}_j(t) = \bfom(\varphi(j, t))^{\perp} \,, \quad \textrm{for } j = 0, 1 \,.
\end{align}
These forces depend on the deformation, and they are nonconservative, i.e., they cannot be derived by variation of a functional. 
Due to this choice, the final model will only partially have a gradient flow structure.

\subsection*{Euler-Lagrange equations.} Equating the sum of variations \eqref{varF}--\eqref{varB} (with forces \eqref{forces}) to zero for arbitrary $(\bfdz,\delta\varphi)$
leads to the Euler-Lagrange equations
\begin{eqnarray} \label{EL1}
\dt \bfz &=& \da\left(P_0 \,\bfom(\varphi)^\perp + \gamma \frac{\da\bfz}{|\da\bfz|}\right)  \,,\\
\beta^2 \dt\varphi &=& P_0 \,\da\bfz \cdot \bfom(\varphi) + \beta^2 \da P_1 \,, \label{EL2}
\end{eqnarray}
for $0<\alpha<1$, and the boundary conditions
\begin{align} \label{BC}
(P_0 - 1 + \gamma) \bfom(\varphi)^\perp + \gamma \frac{\da\bfz}{|\da\bfz|}	 = \da\varphi = 0 \,,\qquad\text{for } \alpha =0, 1 \,.
\end{align}
The second boundary condition is equivalent to $P_1=0$, which is obvious from the representations
\begin{eqnarray}
P_0 &=& \int_{-1/2}^{1/2} \frac{ds}{\da\bfz^\bot \cdot \bfom(\varphi) - \sqrt{12}\,\beta s\,\da\varphi} \,,\label{P0}\\
P_1 &=& \frac{12\,\da\varphi}{\da\bfz^\perp\cdot\bfom(\varphi)} \int_{-1/2}^{1/2} \frac{s^2 ds}{\da\bfz^\bot \cdot \bfom(\varphi) - \sqrt{12}\,\beta s\,\da\varphi} \,. \label{P1}
\end{eqnarray}
When taking the derivative of \eqref{EL1} with respect to $\alpha$, a closed system for $(\da\bfz,\varphi)$ is obtained. Therefore we introduce $\bfu := \da\bfz$ and rewrite 
\eqref{EL1}--\eqref{P1} as
\begin{eqnarray} \label{EL1a}
\dt \bfu &=& \da^2\left(P_0 \,\bfom(\varphi)^\perp + \gamma \frac{\bfu}{|\bfu|}\right)  \,,\\
\beta^2 \dt\varphi &=& P_0 \,\bfu \cdot \bfom(\varphi) + \beta^2 \da P_1 \,, \label{EL2a}
\end{eqnarray}
with the boundary conditions
\begin{align} \label{BCa}
(P_0 - 1 + \gamma) \bfom(\varphi)^\perp + \gamma \frac{\bfu}{|\bfu|}	 = \da\varphi = 0 \,,\qquad\text{for } \alpha =0, 1 \,,
\end{align}
and with 
\begin{eqnarray}
P_0 &=& \int_{-1/2}^{1/2} \frac{ds}{\bfu^\bot \cdot \bfom(\varphi) - \sqrt{12}\,\beta s\,\da\varphi} \,,\label{P0a}\\
P_1 &=& \frac{12\,\da\varphi}{\bfu^\perp\cdot\bfom(\varphi)} \int_{-1/2}^{1/2} \frac{s^2 ds}{\bfu^\bot \cdot \bfom(\varphi) - \sqrt{12}\,\beta s\,\da\varphi} \,. \label{P1a}
\end{eqnarray}

\subsection*{Formal properties.} The right hand side of \eqref{EL1}, \eqref{EL2} is a quasilinear second order operator. The leading order terms can be written in the form
$$
\begin{pmatrix}
k\, \bfom^\perp \otimes \bfom^\perp + \dfrac{\gamma}{|\bfu|^3}\bfu^\perp \otimes \bfu^\perp & \dfrac{\partial P_0}{\partial(\da\varphi)}\bfom^\perp \nl 
\nabla_{\bfu} P_1 & \dfrac{\partial P_1}{\partial(\da\varphi)}    
\end{pmatrix}    
\da^2 \begin{pmatrix} \bfz \\ \varphi \end{pmatrix} \,.
$$
Straightforward computations show that $k >0$ and $\partial P_1/\partial(\da\varphi)>0$. This implies that for $\bfu^\perp\cdot\bfom = \dfrac{1}{\rho(s=0)}> 0$ 
(which is needed for well-posedness anyway) the block diagonal part of the coefficient matrix is positive definite, an indication that the system is parabolic. To make this a more rigorous 
argument the off-diagonal cross-diffusion terms would have to be checked. 
Concerning questions of well-posedness, the form \eqref{EL1a}--\eqref{EL2a} is less useful, since \eqref{EL1a} contains the third order derivative of $\varphi$ with respect to $\alpha$.
However, for the formal computations in the following sections the formulation \eqref{EL1a}--\eqref{P1a} will be convenient.

Integration of \eqref{EL1} with respect to $\alpha$ and using the boundary conditions \eqref{BC} gives
$$
\frac{d}{dt} \int_0^1 \bfz\,d\alpha = (1-\gamma)\bfom(\varphi)^\perp \Big|_{\alpha=0}^1  \,.
$$
This shows that the center of mass of the whole structure moves as long as the outermost filaments are not parallel. When the formulation \eqref{EL1a}--\eqref{P1a} is used,
this equation allows to recover the positions of the filaments, since $\bfz$ can be computed from its average and from $\bfu = \partial_\alpha\bfz$.

The system \eqref{EL1a}--\eqref{P1a} has a family of (trivial) steady states given by
\begin{equation}\label{triv-state}
\varphi = \varphi_0 = \rm{const} \,,\qquad \bfu = -\bfom(\varphi_0)^\perp \,,
\end{equation}
with $P_0=1$, $P_1=0$, and with arbitrary $\varphi_0\in \R$. In this state the filaments cover a rectangle, which does not move. The freedom in choosing $\varphi_0$ is a 
consequence of the rotational symmetry of the problem, which also has a translational symmetry, whence the rectangle can be put in an arbitrary position: 
$z(\alpha) = z_0 -\alpha \bfom(\varphi_0)^\perp$ with arbitrary $z_0\in\R^2$.

\section{Linearization around a trivial steady state}\label{sec:lin}
Without loss of generality, we choose the trivial steady state 
\begin{equation}\label{trivial}
\overline{\bfu} =  \begin{pmatrix}1 \\ 0\end{pmatrix} \,,\qquad \overline\varphi = \frac{\pi}{2} \,,
\end{equation}
of \eqref{EL1a}--\eqref{P1a}. With the ansatz
\begin{align*}
\bfu = \overline{\bfu} + \begin{pmatrix} a \\ b \end{pmatrix} \,, \qquad \varphi = \overline{\varphi} + \psi \,,
\end{align*}
linearization of \eqref{EL1a}--\eqref{P1a} produces two decoupled problems for $a$ and for $(b,\psi)$.
The component $a$ solves the homogeneous Dirichlet problem for the heat equation,
\begin{align} \label{eq:bif1a}
\dt a = \da^2 a \,, \qquad
a(0,t) = a(1,t) = 0 \,,
\end{align}
and thus decays to zero exponentially. For the pair $(b,\psi)$ we obtain the cross diffusion problem
\begin{eqnarray}
\dt b &=& \da^2 (\gamma b - \psi) \,, \label{eq:bif1} \nl
\beta^2 \dt \psi &=& \beta^2 \da^2 \psi + b - \psi \,, \label{eq:bif1p}
\end{eqnarray}
with boundary conditions
\begin{align}
b - \psi = \da \psi = 0\,, \qquad \text{for } \alpha = 0, 1 \,. \label{eq:bif1bdy}
\end{align}
Rotational invariance manifests itself by the invariance with respect to $(b,\psi) \to (b+c,\psi+c)$, $c\in\R$, and the consequential family of trivial steady states 
$(b,\psi)= (\varphi_0-\overline{\varphi})(1,1)$, $\varphi_0\in\R$.

\subsection*{Nontrivial steady states.} With $\mathbf{X}(\alpha) := (b, \psi)$, we write the steady state equations of \eqref{eq:bif1}--\eqref{eq:bif1p} in the form
\begin{align} \label{eq:lin-ss}
\da^2\mathbf{X}(\alpha) = M \mathbf{X}(\alpha) \,, 
\qquad\text{with } M := \frac{1}{\beta^{2}}\begin{pmatrix}
-1/\gamma & 1/\gamma \\ -1 & 1
\end{pmatrix} \,.
\end{align}
The eigenvalues of $M$ are 
\begin{align*}
\lambda_0 = 0 \quad \textrm{and} \quad \lambda_1 = - \frac{1-\gamma}{\beta^2\gamma} \,, 
\end{align*}
with eigenvectors $(1, 1)$ and $(1, \gamma)$, respectively. 
Note that $\lambda_{1} < 0$ since $0 < \gamma < 1$. The general solution of \eqref{eq:lin-ss} is given by
\begin{align*}
\mathbf{X}(\alpha) = (A\sin(\kappa \alpha) + B \cos(\kappa \alpha)) \begin{pmatrix} 1 \\ \gamma \end{pmatrix}
+ (C\alpha + D) \begin{pmatrix} 1 \\ 1 \end{pmatrix} , 
\quad\text{with } \kappa := \sqrt{-\lambda_1} = \frac{1}{\beta} \sqrt{\frac{1-\gamma}{\gamma}} \,,
\end{align*}
and with arbitrary $A,B,C,D\in\R$. Inserting this into the boundary conditions \eqref{eq:bif1bdy} implies that a nontrivial steady state exists if and only if $\sin\kappa=0$ and $\cos\kappa=1$. For reasons explained below, we concentrate only on the smallest choice $\kappa_{0}=2\pi$. The set of nontrivial solutions is given by
\begin{align} \label{eq:lin-X}
b(\alpha) = A(\sin(2\pi \alpha) - 2\pi \gamma \alpha) + D \,,\quad 
\psi(\alpha) = A\gamma(\sin(2\pi \alpha) - 2\pi \alpha) + D \,,\quad A, D \in\R\,.
\end{align}
The constant $D$ corresponds to the trivial steady-state contribution.

Considering $\beta$ as bifurcation parameter, the largest bifurcation point is given by
\begin{align} \label{eq:lin-beta0}
\beta_{0} := \frac{1}{2\pi} \sqrt{\frac{1-\gamma}{\gamma}}\ .
\end{align}

\subsection*{Well-posedness and stability.}
We consider \eqref{eq:bif1a}, \eqref{eq:bif1}--\eqref{eq:bif1bdy} subject to initial conditions
\begin{equation}\label{eq:bif1-IC}
a(\alpha,0) = a_I(\alpha) \,,\qquad b(\alpha,0) = b_I(\alpha) \,,\qquad \psi(\alpha,0) = \psi_I(\alpha) \,.
\end{equation}
In the rest of this section, we abbreviate $\|\cdot\| := \|\cdot\|_{L^2(0,1)}$.

\begin{lemma} \label{lemma1}
Let $\beta>0$, $0<\gamma<1$, and $a_I,b_I,\psi_I\in L^2(0,1)$. Then \eqref{eq:bif1a}, \eqref{eq:bif1}--\eqref{eq:bif1bdy}, \eqref{eq:bif1-IC} has a unique solution $(a,b,\psi)\in C([0,\infty), L^2(0,1))^3$
satisfying
\begin{align}
&\|a(\cdot,t)\| \le e^{-\pi^2 t} \|a_I\| \,,\nonumber\\ 
&\|b(\cdot,t)\| + \|\psi(\cdot,t)\| \le c(\gamma)\exp\left( \frac{\lambda_1(\gamma)}{\beta^2}t\right) \left( \|b_I\| + \|\psi_I\| \right) \,,\label{lin-est}
\end{align}
for $t\ge 0$, with $c(\gamma), \lambda_1(\gamma)>0$.
\end{lemma}

\begin{proof}
The result for $a$ is classical. For the analysis of \eqref{eq:bif1}--\eqref{eq:bif1bdy} we introduce the energy functional
$$
E := \frac{1}{2} \int_0^1 (\gamma(b-\psi)^2 + \psi^2)d\alpha \,,
$$
which is equivalent to $(\|b\| + \|\psi\|)^2$. Along solutions of \eqref{eq:bif1}--\eqref{eq:bif1bdy}, we compute
\begin{eqnarray}
\frac{dE}{dt}  &=& -\int_0^1 \left(\gamma^2 (\da b)^2 + (2 \gamma+1)(\da \psi)^2 - \gamma(2+\gamma)\da b\, \da \psi\right)d\alpha \label{dEdt}\\
&& +\ \frac{1}{\beta^2} \int_0^1 (b-\psi)\left( \psi - \gamma (b-\psi)\right) d\alpha \,.\nonumber
\end{eqnarray}
It is easily seen that the integrand in the first line is positive definite in $(\da b,\da\psi)$, implying
$$
\frac{dE}{dt} \le \frac{2\lambda_1}{\beta^2} E \,,
$$
with an appropriate (large enough) $\lambda_1(\gamma)>0$. This a priori estimate is the essential information to prove that the right hand side of \eqref{eq:bif1} generates a strongly 
continuous semigroup, e.g., by the Hille-Yoshida theorem (we skip the details, which are standard). This proves the existence and uniqueness statement of the theorem.
The inequality \eqref{lin-est} is a consequence of the Gr\"{o}nwall inequality with the constant $c(\gamma)$ resulting from the norm equivalence mentioned above.
\end{proof}

The estimate in the proof is very rough, as the nonnegative term in \eqref{dEdt} has been wasted. Actually, for $\beta>\beta_0$ (before the first bifurcation) we expect the set of trivial
steady states to be stable. We shall prove a result in this direction, assuming $\beta> \beta_0(1+\delta)$ (i.e., $\kappa< \frac{2\pi}{1+\delta}$), $\delta > 0$, in the following. 
In this case, the nullspace of 
\[
\mathcal L(b,\psi) = \left( \da^2(\gamma b - \psi),\ \da^2\psi + \frac{b-\psi}{\beta^2}\right)\,, 
\]
where $D(\mathcal L) = \{(b,\psi): b-\psi=\da\psi=0 \text{ for } \alpha=0,1\}$\,,
is one-dimensional and spanned by $(1,1)$. The adjoint operator $\mathcal L^*$ with respect to the scalar product in $L^2(0,1)^2$ is given by
\[
\mathcal L^*(u,v) = \left( \gamma\da^2 u + \frac{v}{\beta^2}, \da^2(v-u) - \frac{v}{\beta^2}\right)\,,
\]
where $D(\mathcal L^*) = \{(u,v): u=\da((1-\gamma)u - v)=0 \text{ for } \alpha=0,1\}$\,.
It also has a one-dimensional nullspace spanned by
\begin{equation}\label{uv}
(u_\kappa,v_\kappa) := \kappa\left( \sin(\kappa\alpha) - \frac{1-\cos(\kappa\alpha)}{1-\cos\kappa}\, \sin\kappa \,, 
(1-\gamma)\left( \sin(\kappa\alpha) + \frac{\cos(\kappa\alpha)}{1-\cos\kappa}\, \sin\kappa \right)\right) \,.
\end{equation}
Note that the factor $\kappa$ removes the singularity of $v_\kappa$ at $\kappa=0$, and by the assumption $\kappa < \frac{2\pi}{1+\delta}$ we stay away from 
the singularity at $\kappa=2\pi$. Consequentially,
\begin{equation}\label{uv-bound}
|u_\kappa(\alpha)|\,, |v_\kappa(\alpha)| \le c_1(\delta)\,, \quad 0\le \alpha\le 1 \,.
\end{equation}
For solutions $(b,\psi)$ of \eqref{eq:bif1} this implies the conservation law
\begin{equation}\label{cons-law}
\frac{d}{dt} \int_0^1 (bu_\kappa+\psi v_\kappa)d\alpha  = 0 \,.
\end{equation}
Therefore, we expect convergence as $t\to\infty$ to the trivial steady state $(b_\infty,\psi_\infty) = c(1,1)$ with the constant $c$ determined from
$$
c \int_0^1 (u_\kappa+ v_\kappa)d\alpha =  \int_0^1 (b_I u_\kappa+\psi_I v_\kappa) d\alpha \,.    
$$

\begin{lemma}\label{lem:coeff}
Let $0<\kappa<2\pi$ and $(u_\kappa,v_\kappa)$ be given by \eqref{uv}. Then 
$$
\int_0^1 (u_\kappa + v_\kappa)d\alpha > 2(1-\gamma) >0 \,.
$$
\end{lemma}

\begin{proof}
We compute
$$
\int_0^1 (u_\kappa+ v_\kappa)d\alpha = 2\left(2-\gamma - \frac{\kappa\sin\kappa}{2(1-\cos\kappa)}\right) \,.
$$
It is easily checked that $2(1-\cos\kappa)-\kappa\sin\kappa>0$, implying the result.
\end{proof}
After the transformation $(b,\psi) \to (b+b_\infty,\psi+\psi_\infty)$ we may assume from now on
\begin{equation}\label{IC-cond}
\int_0^1 (b_I u_\kappa+\psi_I v_\kappa)d\alpha = 0 \,. 
\end{equation}

\begin{theorem} \label{thm:wp}
Let $0<\gamma<1$ and $\beta\ge \beta_0(1+\delta) = \dfrac{1}{2\pi}\sqrt{\dfrac{1-\gamma}{\gamma}}(1+\delta)$ with $\delta>0$\,. Assume $b_I,\psi_I\in L^2(0,1)$ satisfy \eqref{IC-cond}. Then the solution of \eqref{eq:bif1}--\eqref{eq:bif1bdy}, \eqref{eq:bif1-IC} satisfies
$$
\|b(\cdot,t)\| + \|\psi(\cdot,t)\| \le c(\gamma)\exp\left( \left(\frac{\lambda_1(\gamma)}{\beta^2} - \lambda_2(\gamma,\delta)\right)t\right) \left( \|b_I\| + \|\psi_I\| \right) \,, 
$$
with $\lambda_2(\gamma,\delta)>0$ and with $\lambda_1(\gamma), c(\gamma)$ as in Lemma \ref{lemma1}.
\end{theorem}

\paragraph{\bf Remark.}
\begin{enumerate}[(i)]
\item The theorem shows that for $\beta$ large enough the set of trivial steady states is spectrally stable. Explicit formulas for $\lambda_1(\gamma)$ and 
$\lambda_2(\gamma,\delta)$ could be given. However they would be rather complicated and they would not be optimal, i.e., $\lambda_1/\lambda_2 > \beta_0^2$\,.
\item For general initial data the asymptotic state $(b_\infty,\psi_\infty)$ can be computed by the conservation law \eqref{cons-law}. We are not aware of a corresponding conservation law for the
nonlinear problem. We expect that for $\beta> \beta_0$ solutions of the nonlinear problem converge to a trivial steady state \eqref{triv-state}, however there is no explicit prediction of the 
angle $\varphi_0$ from the initial data.
\end{enumerate}
\begin{proof}[Proof of Theorem \ref{thm:wp}]
As already stated in the proof of Lemma \ref{lemma1}, the quadratic expression of $(\da b,\da\psi)$ in the first integral in \eqref{dEdt} is positive definite, since 
$4\gamma^2(2\gamma + 1)>\gamma^2(2+\gamma)^2$. Therefore the integral can be bounded from below by
$$
c_2(\gamma) \left(\|\partial_\alpha(b-\psi)\|^2 + \|\partial_\alpha \psi\|^2 \right)\,, \qquad \text{with } c_2 >0 \,.
$$
Now we use the Poincar\'e inequalities for functions in $H_0^1((0,1))$ and in $H^1((0,1))$\,:
$$
\|\partial_\alpha(b-\psi)\|^2 + \|\partial_\alpha \psi\|^2 \ge \pi^2 \left(\|b-\psi\|^2 + \|\psi - \overline\psi\|^2 \right) \,,
$$
with the average
$$
\overline\psi := \int_0^1 \psi\,d\alpha \,.
$$
For both Poincar\'e inequalities $\pi^2$ is the optimal constant. Since \eqref{IC-cond} holds for all times, we have
$$
\int_0^1 ((b-\psi)u_\kappa + (\psi-\overline\psi)(u_\kappa+v_\kappa))d\alpha + \overline\psi \int_0^1 (u_\kappa+v_\kappa)d\alpha = 0 \,.
$$
This implies by the Cauchy-Schwarz inequality, by Lemma \ref{lem:coeff}, and by \eqref{uv-bound} that
$$
|\overline\psi| \le \frac{c_1(\delta)}{2(1-\gamma)} \left(\|b-\psi\| + \|\psi - \overline\psi\| \right) \,,
$$
leading to
\begin{align*}
E &= \frac{\gamma}{2}\|b-\psi\|^2 + \frac{1}{2}\|\psi\|^2 = \frac{\gamma}{2}\|b-\psi\|^2 + \frac{1}{2}\|\psi-\overline\psi\|^2 + \frac{1}{2}{\overline\psi}^2 \nl
&\le c_3(\gamma,\delta)\left(\|b-\psi\|^2 + \|\psi - \overline\psi\|^2 \right)  \,.
\end{align*}
Combining our results with the proof of Lemma \ref{lemma1} completes the proof with $\lambda_2 = \dfrac{c_2\pi^2}{c_3}$.
\end{proof}

\section{Normal form reduction close to the bifurcation point} \label{sec:bif}
In this section, we study solutions of \eqref{EL1a}--\eqref{P1a} close to the trivial steady state \eqref{trivial} for values of the bifurcation parameter $\beta$ close to $\beta_0$\,.
Additional to the translation and rotation symmetries mentioned above, the problem is also invariant under the reflection $\alpha\to-\alpha$, $s\to-s$, $\bfu\to -\bfu$, 
$\varphi\to\varphi+\pi$. For this reason we expect a pitchfork bifurcation. The perturbation scaling given below is guided by this expectation. 

Closeness to the trivial steady state and to the bifurcation point will be measured in terms of the small positive parameter
\begin{align*}
\varepsilon := \sqrt{\abs{\beta_0^2 - \beta^2}} \,, \qquad \text{whence}\quad \beta^2 = \beta_0^2 - \sigma \varepsilon^2\,, \quad \text{with} \quad  \sigma := \sign(\beta_0^2 - \beta^2)\,.
\end{align*}
For the solution we make the ansatz
\begin{align} 
\bfu &= \overline{\bfu} + \varepsilon \bfu_1 + \varepsilon^2 \bfu_2 + \varepsilon^3 \bfu_3 + \mathcal{O}(\varepsilon^4)\ , \quad \bfu_j = (a_j, b_j) \,, \nonumber\nl
\varphi &= \overline\varphi + \varepsilon \varphi_1 + \varepsilon^2 \varphi_2 + \varepsilon^3 \varphi_3 + \mathcal{O}(\varepsilon^4) \,.\label{eq:bif1b}
\end{align}
Another preparatory step is the time rescaling $t \mapsto t/\varepsilon^{2}$, whence system \eqref{EL1a}--\eqref{BCa} becomes
\begin{align} 
&\varepsilon^2 \dt\bfu  = \da^2\left( P_0 \bfom(\varphi)^\perp + \gamma \frac{\bfu}{|\bfu|} \right) \,, &0<\alpha< 1 \,,  \nonumber\nl
&\varepsilon^2\beta^2 \dt\varphi = P_0 \bfu \cdot \bfom(\varphi) + \beta^2 \da P_1 \,, &0<\alpha< 1 \,, \label{eq:bif2}\nl
&(P_{0} - 1 + \gamma) \bfom(\varphi)^\perp + \gamma \frac{\bfu}{|\bfu|} = \da\varphi = 0, &\alpha = 0, 1 \,,  \nonumber
\end{align}
with $P_0$ and $P_1$ given by \eqref{P0a}, \eqref{P1a} (as functions of $(\bfu,\varphi,\da\varphi)$).
It is straightforward but computationally tedious to insert the power series ansatz and to re-expand all the terms in these equations. The details will be omitted.
Comparing coefficients of equal powers of $\varepsilon$ will provide equations for the coefficients in the expansion \eqref{eq:bif1b}\,:

\subsection*{$\mathcal{O}(\varepsilon)$-equations.} At $\mathcal{O}(\varepsilon)$ we obtain the steady state version of the linearized system \eqref{eq:bif1a}-\eqref{eq:bif1bdy}
for the unknowns $\bfu_1$ and $\varphi_1$. These equations have already been solved in the previous section, implying $a_1 = 0$ and $(b_1, \varphi_1)$ given by \eqref{eq:lin-X}:
\begin{align} \label{eq:lin} 
b_1(\alpha,t) = A(t)(\sin(2\pi \alpha) - 2\pi \gamma \alpha) +D(t)\,,\  \varphi_1(\alpha,t) = A(t)\gamma(\sin(2\pi \alpha) - 2\pi \alpha ) +D(t) \,.
\end{align}
The only goal of the following computations for the higher order terms is to determine the dynamics of $A(t)$ and $D(t)$.

\subsection*{$\mathcal{O}(\varepsilon^2)$-equations.} At any higher order $\varepsilon^j$ we obtain inhomogeneous versions of the steady state linearized system for $(\bfu_j,\varphi_j)$
with inhomogeneities depending only on coefficients of order up to $j-1$. The problem for $a_2$ reads
\begin{align} \label{eq:lin-Xhat1}
\begin{array}{l}
\partial_{\alpha}^{2} \left(a_2 + b_1\varphi_1 - \frac{1}{2} \gamma b_1^2 - \beta_0^2 (\da\varphi_1)^2 \right) = 0 \,, \nl
a_2 + b_1\varphi_1 - \frac{1}{2} \gamma b_1^2  = \frac{1}{2} (1-\gamma) \varphi_1^2 \,, \qquad \alpha = 0, 1 \,,
\end{array}
\end{align}
and for $(b_2,\varphi_2)$ we obtain the homogeneous problem (as for $(b_1,\varphi_1)$), using $a_1=0$. This is the reason, why the equations up to third order have to be considered. 
The solution $(b_2,\varphi_2)$ will not be needed in the following. Noting that all inhomogeneities in \eqref{eq:lin-Xhat1} are quadratic in the first order coefficients, the solution can be
computed explicitly and written as
\begin{align*}
a_2(\alpha,t) = A(t)^2 h_1(\alpha) + A(t) D(t) h_2(\alpha) - \frac{D(t)^2}{2}\,,
\end{align*}
with appropriate functions $h_1$ and $h_2$\,.

\subsection*{$\mathcal{O}(\varepsilon^3)$-equations.} Only the problem for $(b_3,\varphi_3)$ is needed:
\begin{align*} 
&\dt b_1 = \da^2 \left(\gamma b_3 - \varphi_3 - a_2 (\gamma b_1 - \varphi_1) + b_1\varphi_1^2 - \frac{1}{3} \varphi_1^3 
- \frac{\gamma}{2} b_1^3 - \beta_0^2 \varphi_1 (\da \varphi_1)^2  \right) \,, \nl
&\beta_0^2\dt \varphi_1 = \beta_0^2 \da^2 \varphi_3 + b_3 - \varphi_3 - \sigma \da^2 \varphi_1 - a_2 b_1 
+ b_1 \varphi_1^2 - b_1^2 \varphi_1 - \frac{1}{3} \varphi_1^3 \nl
&\qquad \qquad + \beta_0^2 (b_1 - \varphi_1) (\da\varphi_1)^2 + \beta_0^2\da\left( \left( \varphi_1^2 -2a_2 - 2b_1 \varphi_1
+ \frac{9}{5}\beta_0^2 (\da\varphi_1)^2 \right) \da\varphi_1 \right) \,,  \nl
&\gamma (b_3 - \varphi_3) - a_2 (\gamma b_1 - \varphi_1) + b_1\varphi_1^2 - \frac{1}{3} \varphi_1^3 - \frac{\gamma}{2} b_1^3 
- \beta_0^2 \varphi_1 (\da\varphi_1)^2 = \frac{1}{6}(1-\gamma) \varphi_1^3 \,, \qquad \alpha = 0, 1 \,, \nl
&\da\varphi_3 = 0 \,,  \qquad \alpha = 0, 1 \,.
\end{align*}
Using the formulas for $b_1,\varphi_1,a_2$, the differential equations can be written in compact form as
\begin{align} \label{3rd} 
\dot A \begin{pmatrix} \tilde b \\\beta_0^2\tilde\varphi\end{pmatrix} 
+ \dot D \begin{pmatrix} 1 \\ \beta_0^2\end{pmatrix} 
&= \mathcal{L}_0(b_3,\varphi_3) - \sigma A \begin{pmatrix} 0 \\ \tilde\varphi'' \end{pmatrix}
+ A^3 \begin{pmatrix} h_3 \\ h_4\end{pmatrix} \notag\\
&+ A^2 D \begin{pmatrix} h_5 \\ h_6\end{pmatrix} + A D^2 \begin{pmatrix} h_7 \\ h_8\end{pmatrix} + D^3 \begin{pmatrix} 0 \\ 1/6 \end{pmatrix}\,, 
\end{align}
where $\dot{}$ denotes time derivative, with appropriate functions $\tilde b \,, \tilde\varphi$\,, and $h_j$ for $j=3, 4, \ldots, 8$, and 
\begin{align} \label{eq:lin-l1}
\mathcal{L}_0 (b, \varphi) :=
\begin{pmatrix}
\da^2 (\gamma b - \varphi) \nl \beta_0^2 \da^2 \varphi + b - \varphi
\end{pmatrix} \,,
\end{align}
essentially the operator $\mathcal{L}$ from the previous section with $\beta=\beta_0$. An equation for $A$ and $D$ will be derived as a solvability condition for 
\eqref{3rd}. With the domain 
$$
\mathcal{D}(\mathcal{L}_0) := \left\{(b, \varphi): b - \varphi = \da \varphi = 0 \text{ for } \alpha = 0, 1 \right\}  \,,
$$ 
the adjoint with respect to $L^2(0,1)^2$ reads
\begin{align*}
\mathcal{L}_0^*(u,v) := 
\begin{pmatrix}
\gamma \da^2 u + v \nl  \da^2(\beta_0^2 v - u) - v
\end{pmatrix} \,,
\end{align*}
with domain $\mathcal{D}(\mathcal{L}_0^*) := \left\{(u,v):\,u = \da(4\pi^2\gamma u - v) = 0 \text{ for } \alpha = 0, 1\right\}$. 
Since we are at the bifurcation point $\beta=\beta_0$, the nullspaces of $\mathcal{L}_0$ and $\mathcal{L}_0^*$ are two-dimensional.
In particular, $(u,v)\in \mathcal{N}(\mathcal{L}_0^*)$ if and only if there exist $A^*,D^*\in\R$ such that
\begin{align*}
\begin{pmatrix} u\\ v \end{pmatrix} = A^*\sin(2\pi \alpha) \begin{pmatrix} 1 \\ 4\pi^{2} \gamma \end{pmatrix} 
+ D^*  \begin{pmatrix} \cos(2\pi \alpha) - 1 \\ 4\pi^2 \gamma \cos(2\pi\alpha)\end{pmatrix} \,.
\end{align*}
An ODE system for $A$ and $D$ is now derived by taking the scalar product of \eqref{3rd} with $(u,v)$:
\begin{align} 
\dot A \int_0^1 (\tilde b u + \beta_0^2 \tilde\varphi v)d\alpha + \dot D \int_0^1 (u+\beta_0^2 v)d\alpha = -\sigma A \int_0^1 \tilde\varphi'' v d\alpha + A^3 \int_0^1 (h_3 u + h_4 v) d\alpha \notag \\
+\ A^2 D \int_0^1 (h_5 u + h_6 v) d\alpha + AD^2 \int_0^1 (h_7 u + h_8 v) d\alpha \,. \label{eq:sp}
\end{align}

It is left to expand and evaluate the integrals in \eqref{eq:sp}. Starting with its left-hand side, we compute
\begin{align} \label{eq:ad1}
\dot A \int_0^1 (\tilde b u + \beta_0^2 \tilde\varphi v)d\alpha + \dot D \int_0^1 (u+\beta_0^2 v)d\alpha = \left( \frac{1}{2}\left(1 + 5\gamma - 3 \gamma^2 \right) A^* + \pi \gamma D^* \right) \dot{A} - D^* \dot{D} \,.
\end{align}
On the right-hand side of \eqref{eq:sp}, the first integral becomes
\begin{align} \label{eq:ad2}
-\sigma A \int_{0}^{1} \tilde{\varphi}'' v d\alpha = 8 \pi^4 \gamma^2 \sigma A^* A \,.
\end{align}
In the following, we summarize the results of straightforward yet lengthy computations for the remaining integrals of \eqref{eq:sp}. Parts of the derivation have been carried out manually and cross-checked with \texttt{Mathematica} for consistency. After integration by parts, we obtain boundary terms, which reduce to 
\begin{align}
-\frac{(1-\gamma)}{6} (\da u) \varphi_1^3 
\Big|_0^1 
= \frac{2\pi^2 \gamma}{3} (1-\gamma) (4\pi^2 \gamma^2 A^2 - 6\pi \gamma AD + 3D^2) A^* \,,
\label{eq:ad2a}
\end{align}
where we applied the boundary conditions for the $\mathcal{O}(\varepsilon^3)$-equations, and that $(u, v) \in \mathcal{D}(\mathcal{L}_{0}^*)$\,. The remaining terms are further simplified using the adjoint equations,
\begin{align*}
\int_0^1 \Bigg[&\left(a_2 \varphi_1 + b_1 \varphi_1^2 - \frac{1}{3} \varphi_1^3 - \frac{\gamma}{2} b_1^3 - \beta_0^2 \varphi_1 (\da \varphi_1)^2 \right) \da^2 u
+ \Bigg(b_1 \varphi_1^2 - b_1^2 \varphi_1 - \frac{1}{3} \varphi_1^3 \notag \nl
&+\ \beta_0^2 (b_1 - \varphi_1) (\da \varphi_1)^2 +  \beta_0^2 \da \left( \varphi_1^2 - 2a_2 -2b_1 \varphi_1 + \frac{9}{5} \beta_0^2 (\da \varphi_1)^2 \right) (\da \varphi_1) \Bigg) v \Bigg] \diff \alpha \,, 
\end{align*}
which we compute as
\begin{align}
\frac{\pi^2 \gamma}{12} (1-\gamma)\left( (9+36\gamma - 6 (11+4\pi^2)\gamma^2 - 4(21+4\pi^2)\gamma^3 ) A^2 + 48 \pi \gamma AD - 24 D^2 \right) A^* A \,.
\label{eq:ad5}
\end{align}

Taking the sum of \eqref{eq:ad2}, \eqref{eq:ad2a}, and \eqref{eq:ad5}, the terms involving $D$ drop out. 
Equating this sum with \eqref{eq:ad1}, we get an ODE for our unknowns $A$ and $D$,
\begin{align*}
&\left( \frac{1}{2}\left(1 + 5\gamma - 3 \gamma^2 \right) A^* + \pi \gamma D^* \right) \dot{A} - D^* \dot{D} \nl
&= \frac{\pi^2 \gamma}{12} \left( (1-\gamma)(1-2\gamma) (9+ 54\gamma +(42+8\pi^2)\gamma^2) A^2 + 96 \pi^2 \gamma \sigma\right) A^* A \,. 
\end{align*}
Finally, we compare the coefficients of $A^*$ and $D^*$ above to write equations for $A$ and $D$, respectively. In particular,
\begin{align} \label{eq:ad7}
\dot{A} = A (\kappa_1(\gamma) \sigma - \kappa_2(\gamma) A^2) \,,
\end{align}
where
\begin{align*}
\kappa_1(\gamma) := \frac{16 \pi^4 \gamma^2}{1+5\gamma - 3\gamma^2} \,, \quad
\kappa_2(\gamma) := \frac{-\pi^2 \gamma (9+ 54\gamma +(42 + 8 \pi^2) \gamma^2)}{6(1+5\gamma-3\gamma^2)}\ (1-\gamma)(1-2\gamma) \,.
\end{align*}
Observe that $\kappa_1(\gamma)> 0$ for $0 < \gamma < 1$. Recalling that $\gamma = (\mu^T / (f_\text{ref} + \mu^T)) \in (0,1)$ measures the relative strength of tension against boundary forces, the sign of $\kappa_2$ depends on the value of $\gamma$ as follows:
\begin{align*}
\begin{cases}
\kappa_2 < 0 &\text{if } 0 < \gamma < \tfrac{1}{2} \quad (\mu^T < f_\text{ref}) \,, \nl
\kappa_2 = 0 &\text{if } \gamma = \tfrac{1}{2} \qquad \quad (\mu^T = f_\text{ref}) \,, \nl
\kappa_2 > 0 &\text{if } \tfrac{1}{2} < \gamma < 1 \quad (\mu^T > f_\text{ref}) \,. \nl
\end{cases}
\end{align*}
For clarity in the following discussion, we refer to the regimes $0 < \gamma < \tfrac{1}{2}$ and $\tfrac{1}{2} < \gamma < 1$ as the \textbf{low-tension} and \textbf{high-tension} cases, respectively. On the other hand,
\begin{align*} 
\dot{D} = \pi \gamma \dot{A} = \pi\gamma A (\kappa_1(\gamma) \sigma - \kappa_2(\gamma) A^2) \,,
\end{align*}
indicating that $D$ does not evolve independently but simply follows the dynamics of $A$, consistent with its interpretation as the trivial steady-state contribution.

\subsection*{Bifurcation regimes} Depending on the value of $0<\gamma<1$, two regimes arise. A summary of these results can be found in Table \ref{table:bif}.\\

\paragraph{\bf Case 1. $\frac{1}{2}<\gamma<1$ (high tension).} Equation~\eqref{eq:ad7} is the normal form of the supercritical pitchfork bifurcation. For $\beta > \beta_0$\,, the trivial steady state is stable, whereas for $\beta < \beta_0$\,, stability is transferred to the bifurcating steady states
\begin{eqnarray}
\bfu &=& \begin{pmatrix} 1\\ 0 \end{pmatrix} \pm \varepsilon \left( \sqrt{\frac{\kappa_{1}}{\kappa_{2}}} \begin{pmatrix} 0\\ \sin(2\pi\alpha) - 2\pi \gamma \alpha \end{pmatrix} + C \right) , \label{eq:ad8a} \nl
\varphi &=& \frac{\pi}{2} \pm \varepsilon \left( 
\sqrt{\frac{\kappa_{1}}{\kappa_{2}}}\ \gamma (\sin(2\pi\alpha) - 2\pi \alpha) + C \right), \label{eq:ad8b} 
\end{eqnarray}
for some constant $C$ depending on the initial data.\\

\paragraph{\bf Case 2. $0<\gamma<\frac{1}{2}$ (low tension).} Equation \eqref{eq:ad7} is the normal form of the subcritical pitchfork bifurcation. In this case, for $\beta > \beta_0$\,, the trivial steady state is stable and there are two unstable equilibria with $A = \pm\sqrt{-\kappa_1/\kappa_2}$\,. For 
$\beta < \beta_0$\,, the trivial steady state is unstable and the bifurcation analysis does not provide a candidate for long time behavior. This will be investigated numerically in the following section.\\

\paragraph{\bf Remark.} For the degenerate case $\gamma=\frac{1}{2}$ (equal tension and boundary forces), higher-order nonlinear terms would be needed to completely determine the qualitative behavior. Since the super- and sub-critical regimes provide sufficient insight into the dynamics of our system, we no longer compute the said terms.\\

\begin{table}[ht] \centering 
\caption{Behavior of steady states depending on the dimensionless parameters $\beta$ and $\gamma$}

\renewcommand{\arraystretch}{1.5}
\begin{tabular}{|c | c | c |} \hline
\ & $\frac{1}{2}< \gamma<1$ (supercritical) &$0<\gamma<\frac{1}{2}$ (subcritical) 	
\\ \hline
\begin{tabular}{@{}c@{}}
$\beta>\beta_0$ 
\end{tabular}
& $A=0$ is stable 
&\begin{tabular}{@{}l@{}}
$A=0$ is stable\\
$A = \pm \sqrt{-\kappa_1/ \kappa_2}$ is unstable
\end{tabular} 
\\ \hline
\begin{tabular}{@{}c@{}}
$\beta<\beta_0$ 
\end{tabular}
&\begin{tabular}{@{}l@{}} 
$A=0$ is unstable\\ $A = \pm \sqrt{\kappa_1/\kappa_2}$ is stable
\end{tabular}
&$A=0$ is unstable 
\\ \hline
\end{tabular}
\label{table:bif}
\end{table}

\section{Numerical Simulations\label{sec:num}}
To solve \eqref{EL1}--\eqref{P1} numerically, we employ the method of lines. This means that we first discretize in space to generate a system of ODEs and then solve the corresponding initial value problem in time using a standard ODE integrator. Here, space was discretized into $n = 40$ nodes to produce 120 ODEs, explained further below. Those ODEs were integrated using \texttt{ode15s}, an adaptive time integrator of \texttt{Matlab}. Non-default values used include an absolute and relative tolerance of $10^{-6}$.

For convenience, let us rewrite \eqref{EL1}-\eqref{P1} in terms of fluxes,
\begin{align} \label{eq:flux1}
\begin{cases}
\dt \bfz = \da \vec{f}(\bfz, \varphi) \,, &\vec{f}(\bfz, \varphi)\Big|_{\alpha = 0, 1} = (1-\gamma) \bfom (\varphi)^{\perp} \Big|_{\alpha = 0, 1} \,, \nl
\dt\varphi = \da g(\bfz, \varphi) + s(\bfz, \varphi) \,,  &g(\bfz, \varphi)\Big|_{\alpha = 0, 1}  = 0 \,, 
\end{cases}
\end{align}
where 
\begin{align} \label{eq:flux2}
\vec{f}(\bfz, \varphi) := P_0 \,\bfom(\varphi)^\perp + \gamma\ \frac{\da\bfz}{|\da\bfz|}\,, \quad 
\quad g(\bfz, \varphi) = P_1\,, \quad \text{and} \quad
s(\bfz, \varphi) := \frac{P_0 \,\da\bfz \cdot \bfom(\varphi)}{\beta^2} \,,
\end{align}
with $P_0$ and $P_1$ defined in \eqref{P0} and \eqref{P1}, respectively. Note that $g(\bfz, \varphi)|_{\alpha = 0, 1}  = 0$, for normal cases, is equivalent to $\da \varphi|_{\alpha=0,1}  = 0$. 

We take advantage of the structure of \eqref{eq:flux1} by replacing the $\alpha$-derivatives with finite volume approximations. To this end, the spatial domain $[0, 1]$ is subdivided into cells of size $\Delta \alpha = 1/n$ for $n\in \mathbb{N}$. The corresponding $(n+1)$ cell edges are then given by 
\[
\alpha_i := i/n\,, \quad \text{for } i = 0, \ldots, n\,,
\]
while the $n$ cell centers are 
\[ \alpha_{i+ 1/2} := (\alpha_{i}+\alpha_{i+1})/2\, \quad \text{for } i=0, \ldots, n-1\,. \]
In this strategy, we track the approximate variable values at the cell centers, and not at the cell edges:
\[ \bfz_{i+1/2} := \frac{1}{\Delta \alpha} \int_{\alpha_i}^{\alpha_{i+1}} \bfz(\alpha) d\alpha \approx \bfz(\alpha_{i+1/2})\,, \]
and 
\[ \varphi_{i+1/2} := (1/\Delta\alpha) \int_{\alpha_i}^{\alpha_{i+1}} \varphi(\alpha) d\alpha \approx \varphi(\alpha_{i+1/2})\,, \]
for $i=0,1, \ldots, n\,$. 

Away from the cell edges, we discretize \eqref{eq:flux1}--\eqref{eq:flux2} as follows:
\begin{eqnarray}
\dt \bfz_{i+1/2} &=& (\vec{f}_{i+1} - \vec{f}_i)/\Delta \alpha \,, \label{eq:flux3a}
\\
\dt \varphi_{i+1/2} &=& (g_{i+1}-g_i)/\Delta \alpha + s_{i+1/2} \,, \label{eq:flux3b}
\end{eqnarray}
for $i = 1, \ldots, n-2\,$. Standard averaging and centered differences are used for the unknowns 
\begin{eqnarray*}
\varphi_{i} &=& \left(\varphi_{i+1 / 2}+\varphi_{i-1 / 2}\right) / 2 \,, \nl
\partial_{\alpha} \varphi_{i} &=& \left(\varphi_{i+1 / 2}-\varphi_{i-1 / 2}\right) / \Delta \alpha \,, \nl
\partial_{\alpha} \mathbf{z}_{i} &=& \left(\mathbf{z}_{i+1 / 2}-\mathbf{z}_{i-1 / 2}\right) / \Delta \alpha\,,
\end{eqnarray*}
inside the fluxes $\vec{f}_i := \vec{f}(\bfz_i, \varphi_i)\,$ and $g_i := g(\bfz_i, \varphi_i)\,$, as well as the discrete pressure terms $P_{0, i} := P_0(\bfz_i, \varphi_i)\,$ and $P_{1, i}:= P_1(\bfz_i, \varphi_i)\,$. For the source term in \eqref{eq:flux3b}, we set
\begin{align} \label{eq:flux4}
s_{i+1/2} := (s_i + s_{i+1})/2\,, \qquad \text{where } s_i := s(\bfz_i, \varphi_i)\,,
\end{align}
for reasons stated below.

At the cell edges, the fluxes are naturally defined as
\begin{align*}
\vec{f}_{0}=(1-\gamma) \bfom \left(\varphi_{0}\right)^{\perp}\,, \quad \vec{f}_{n}=(1-\gamma) \bfom \left(\varphi_{n}\right)^{\perp}\,, 
\qquad g_{0}=0\,, \quad g_{n}=0\,.
\end{align*}
To estimate $\varphi_{0}\,$, we introduce a quadratic interpolant $q = q(\alpha)$ satisfying
\begin{align*}
\partial_{\alpha} q(\alpha=0)=\partial_{\alpha} \varphi_{0}=0\,, \quad
q(\alpha=\Delta \alpha / 2)=\varphi_{1 / 2}\,, \quad 
q(\alpha= 3 \Delta \alpha / 2)=\varphi_{3 / 2}\,,
\end{align*}
and similarly for $\varphi_n\,$. These lead to the approximations
\begin{align*}
\varphi_{0}=q(\alpha=0) = \frac{9}{8}\ \varphi_{1 / 2} - \frac{1}{8}\ \varphi_{3 / 2} \,, 
\quad \text{and} \quad
\varphi_{n}=q(\alpha=1)= \frac{9}{8}\ \varphi_{n-1 / 2} -  \frac{1}{8}\ \varphi_{n-3 / 2}\,.
\end{align*}

It is left to estimate the source terms $s_{1/2}$ and $s_{n-1/2}$ in the leftmost and rightmost cells in the domain, respectively. Because methods that utilize exact boundary condition information tend to do better in general, we incorporate such information by noting that, at the left,
\[
P_{0,0}\ \bfom\left(\varphi_{0}\right)^{\perp}+\gamma \frac{\partial_{\alpha} \mathbf{z}_{0}}{\left|\partial_{\alpha} \mathbf{z}_{0}\right|}=(1-\gamma) \bfom\left(\varphi_{0}\right)^{\perp} \,,
\]
scalar multiplication with $\bfom\left(\varphi_{0}\right)$ yields $\partial_{\alpha} \mathbf{z}_{0} \cdot \boldsymbol{\omega}\left(\varphi_{0}\right)=0$\,,
so that 
\[
s_0 = \frac{1}{\beta^2}\ P_{0,0}\ \partial_{\alpha} \mathbf{z}_0 \cdot \boldsymbol{\omega}\left(\varphi_{0}\right) = 0\,.
\]
Similarly, $s_n=0$ and our choice in \eqref{eq:flux4} allows us to incorporate this boundary information consistently into our method while simultaneously providing estimates for $s_{1/2}$ and $s_{n-1/2}$\,.

\begin{example}[Small perturbation from the trivial steady state] \label{ex:small}
Consider the trivial steady state
\begin{align*} 
\bar{\bfz}(\alpha) = \begin{pmatrix}
\bar{\bfz}_1(\alpha) \\ \bar{\bfz}_2(\alpha)
\end{pmatrix}= \begin{pmatrix}
\alpha \\ 0
\end{pmatrix} \quad \text{and} \quad \bar{\varphi} = \frac{\pi}{2} \,, \qquad \alpha \in [0,1]\,.
\end{align*} 
For an initial condition, let us introduce a small perturbation of $(\bar{\bfz}, \bar{\varphi})$, namely,
\begin{align} \label{eq:myic1}
\bfz_I(\alpha) = \bar{\bfz}(\alpha) + 0.1 h(\alpha) = \begin{pmatrix}
\alpha + 0.1 h(\alpha) \\ 0.1 h(\alpha)
\end{pmatrix} \quad \text{and} \quad \varphi_I(\alpha) = \frac{\pi}{2} + 0.01 h(\alpha) \,,
\end{align}
for $\alpha\in[0,1]$, where $h(\alpha)$ is infinitely smooth with $h(0)=h(1)=0$ and $h'(0)=h'(1)=0$, ensuring compatibility with the boundary conditions and avoiding artificial boundary effects, and $\displaystyle \max_{\alpha \in [0,1]} h(\alpha) = 1$. A convenient choice is
\begin{align*}
h(\alpha) = \left\{\begin{array}{cc} \exp\left(\frac{-1}{(\alpha-1/7)(6/7-\alpha)}+\frac{196}{25}\right) & \alpha \in (1/7,6/7)\,, \\
0 & \rm{otherwise.} \end{array} \right.
\end{align*}
In the following, we evolve the perturbed initial condition \eqref{eq:myic1} for parameter regimes chosen to reflect the stability predictions summarized in Table \ref{table:bif}. Specifically, we vary $\beta$ across the bifurcation value $\beta_0$ and consider both high- and low-tension regimes (corresponding to $\gamma = 3/4$ and $\gamma=1/4$, respectively). Each subexample illustrates the long-time behavior of the filament strip and its relation to the stability or instability of the trivial and nontrivial steady states.
\end{example}

\begin{subexample}[$\beta> \beta_0$ -- thick band] 
As predicted by Table \ref{table:bif}, the trivial steady state is stable for $\beta > \beta_0$\,. Let $\gamma = 3/4$ and $\beta =1.01\beta_0$\,, so that $\beta_{0} = 0.0919$ and $\beta = 0.1011$. We plotted the evolution of the filament strip in Figure \ref{fig:0101} for times $t=0.001,0.01,0.1,1,10,100$. There is some initial equilibration as local perturbations give way to a more global curved shape. However, as time increases, the filaments become parallel. The entire strip moves up and very little to the right until it settles to a trivial equilibrium and the strip stops moving.

When $\gamma = 1/4$ and $\beta =1.01\beta_0$\,, the trivial steady state is again stable.
The initial dynamics are similar to those shown in Figure \ref{fig:0101}, with the strip first becoming crescent-shaped before relaxing to the trivial equilibrium. Quantitatively, relaxation occurs faster in this low-tension case: the strip reaches 99\% of its steady-state vertical displacement in approximately 54 time units, compared to approximately 77 time units for $\gamma = 3/4$. The filament strip also travels further north before stopping completely. We omit the corresponding plots since the behavior is qualitatively similar to that shown in Figure \ref{fig:0101}.

\begin{figure}[ht]
\centering
\includegraphics{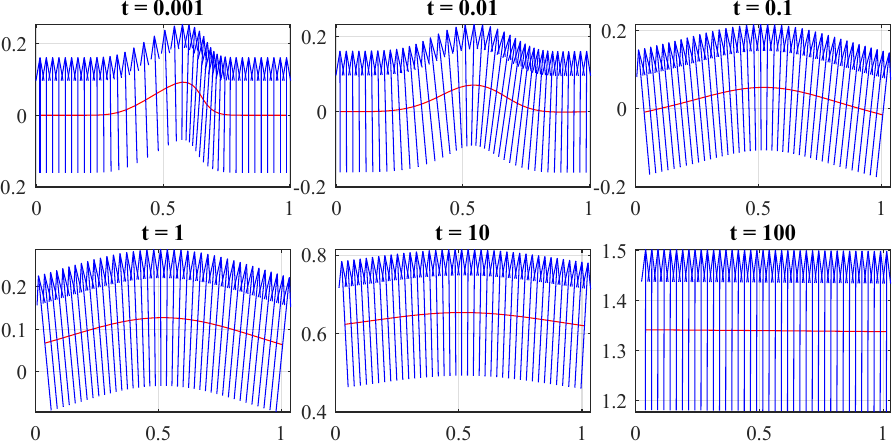}
\caption{Evolution of filament strip in Example 1.1 with $\gamma = 3/4$ and $\beta=1.01\beta_0$ (high tension, thick band). The trivial steady state is stable.} 
\label{fig:0101}
\end{figure}
\end{subexample}

\begin{subexample}[$\gamma = 3/4\,,$ $\beta =0.99\beta_0$ 
-- high tension, thin band]
The first few seconds of strip evolution differ very little from the previous example.  However, at later times, as shown in Figure \ref{fig:0102}, the crescent shape persists  with a relatively small constant rightward velocity compared to the much faster constant upward velocity of the whole structure. This corresponds to the bifurcating steady state \eqref{eq:ad8a}-\eqref{eq:ad8b}, predicted by the linear stability and bifurcation analysis, which we expect to be stable for the current choice of parameters. 

\begin{figure}[ht]
\centering
\includegraphics{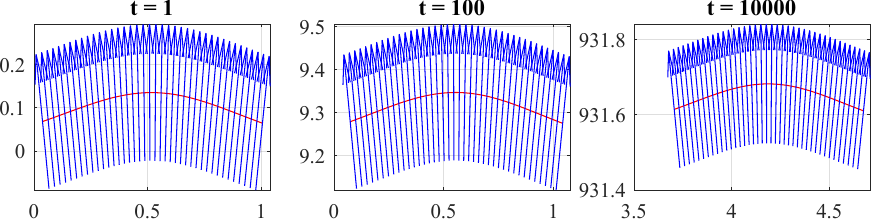}
\caption{Evolution of filament strip in Example 1.2 with $\gamma = 3/4$ and $\beta=0.99\beta_0$ (high tension, thin band). Stability is transferred to the bifurcating steady state \eqref{eq:ad8a}-\eqref{eq:ad8b}.}
\label{fig:0102}
\end{figure}
\end{subexample}

\begin{subexample}[$\gamma = 1/4\,,$ $\beta =0.99\beta_0$ -- low tension, thin band] \label{ex:small4}
Linear theory only predicts that the trivial steady state is unstable. However, as shown in Figure \ref{fig:0104}, the numerical simulations reveal that the system evolves toward a fan-shaped filament strip (which may be viewed as an extreme version of the crescent shape observed in Figure \ref{fig:0102}) that travels predominantly northward and slightly westward at an approximately constant velocity.

The persistence of this traveling configuration suggests the existence of a bifurcating steady state beyond the loss of stability of the trivial equilibrium. Moreover, since this behavior is observed for parameter values near and beyond the bifurcation point, the numerical results indicate the possible presence of a bistable region, even for values $\beta > \beta_0$. We explore this phenomenon further in Example \ref{ex:bi}.

\begin{figure}[ht]
\centering
\includegraphics{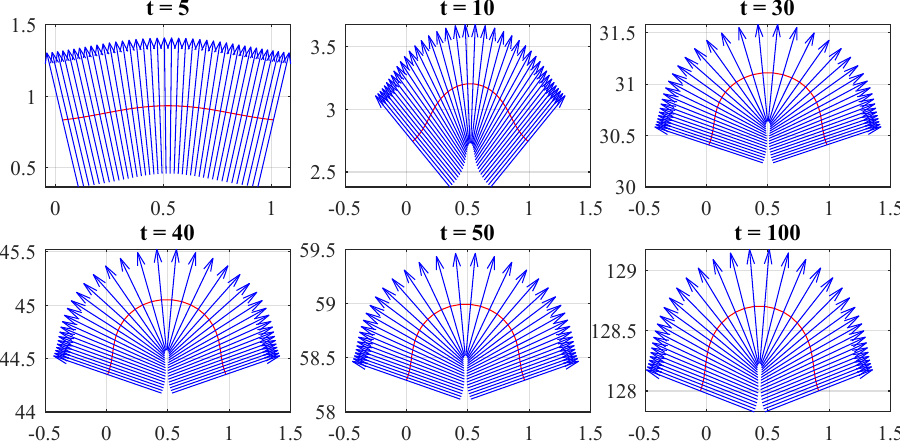}
\caption{Evolution of filament strip in Example 1.3 with $\gamma = 1/4$ and $\beta=0.99\beta_0$ (low tension, thin band).}
\label{fig:0104}
\end{figure}
\end{subexample}

In the examples  above, long-time behavior consists of configurations that persist in shape while translating through space. To prepare for the subsequent discussion, we clarify our use of the term \textit{steady state} in the context of our numerical simulations. Many of the long-time solutions observed are not stationary in the laboratory frame but instead correspond to traveling wave type solutions, in which the filament strip maintains a fixed shape while translating at an approximately constant velocity. In a co-moving reference frame, such solutions correspond to steady states of the system. We therefore refer to both stationary configurations and traveling configurations with constant shape and velocity as \textit{steady states}, unless stated otherwise.

\begin{example}[Subcritical bifurcation -- bistable region for low tension, thick band] \label{ex:bi}
Building on the traveling-wave behavior observed in Example \ref{ex:small4}, we now explore a parameter regime where bistability arises. We set $\gamma=1/4$ and $\beta =1.01\beta_0$\,. According to Table \ref{table:bif}, the trivial steady state is expected to be stable. However, if the initial condition is sufficiently far from the trivial equilibrium, the system can instead evolve toward a bifurcating steady state.

To illustrate this, we compare the evolution of filament strips starting from two comparable initial conditions:
\begin{align} \label{eq:myic2a}
\bfz_I(\alpha) = \begin{pmatrix}
\alpha \\ -\cos(1.134\pi \alpha)
\end{pmatrix} \quad \text{and} \quad \varphi_I(\alpha) = \frac{\pi}{2} \,,
\end{align}
and
\begin{align} \label{eq:myic2b}
\bfz_I(\alpha) = \begin{pmatrix}
\alpha \\ -\cos(1.135\pi \alpha)
\end{pmatrix} \quad \text{and} \quad \varphi_I(\alpha) = \frac{\pi}{2} \,.
\end{align}
In both cases, only the vertical displacement of the filaments is perturbed. The results are plotted side-by-side in Figure \ref{fig:0201}, with \eqref{eq:myic2a} as initial condition on the left and \eqref{eq:myic2b} on the right. 

By $t=10$, both strips have rotated to point southwest. Beyond $t=30$, the left strip relaxes toward a trivial steady state: the filaments become more parallel, the center-of-mass curve straightens, and the motion slows to a stop. In contrast, the right strip evolves into the fan-shaped traveling configuration observed in Example \ref{ex:small4}, eventually translating southwest at an approximately constant velocity. 

This demonstrates the bistable nature of the system for $\beta \gtrsim \beta_0$\, (and $0< \gamma < \frac{1}{2}$). The long-time outcome depends sensitively on the initial condition, with nearby starting positions leading either to the trivial steady state or to a traveling, bifurcating steady state reminiscent of the fan-shaped filament strip.

\begin{figure}[ht] 
\includegraphics{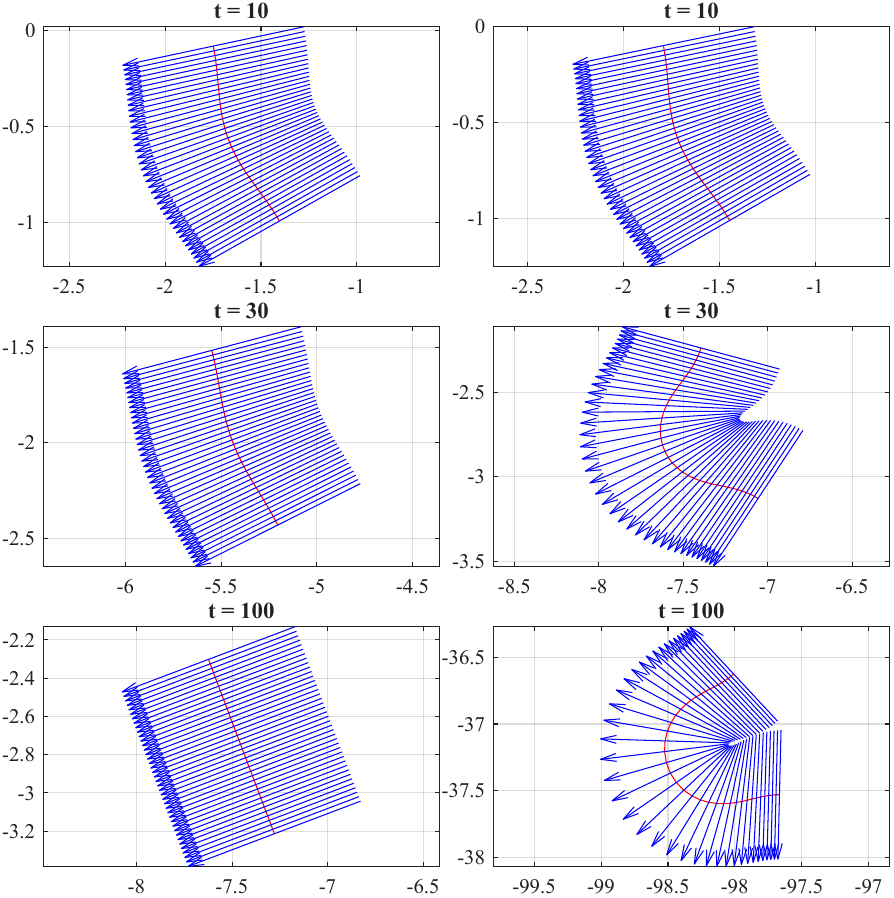}
\caption{Subcritical bifurcation showing bistability for $\gamma=1/4$ and $\beta=1.01\beta_0$. Left: Relaxation to the trivial steady state from \eqref{eq:myic2a}. Right: Evolution toward a bifurcating steady state from \eqref{eq:myic2b}.}
\label{fig:0201}
\end{figure}
\end{example}

\subsection*{Long-time behavior of solutions across $\beta$}
Following the simulations above, we examine how the long-time dynamics of the filament ensemble vary with $\beta$. We plot final velocities to identify trends in stable trivial states, bifurcating steady states, and more complex dynamics in both high- and low-tension regimes.

When solution branches were discovered in the bifurcation space, we used continuation and adaptive stepping with respect to $\beta$ to find additional solutions. This strategy found additional steady state solutions by solving (using Matlab's fsolve \cite{matlab}) the nonlinear system of equations resulting from setting the time derivatives in \eqref{eq:flux3a} equal to a constant velocity, in \eqref{eq:flux3b} equal to zero, and constraining $\frac{1}{n}\sum_{i=1}^n\bfz_i = 0$. While this strategy allows consideration of both stable and unstable steady states without requiring time integration, it can only be used for steady states that exhibit rigid body translation without rotation. When other long-term behaviors were encountered (e.g. spinning), we uniformly sampled those values of $\beta$ and, for each $\beta$\,, integrated over a relatively long time period, $[0,1000]$, to obtain more information about typical long-term behavior.

Many solutions are very close to steady states, either stationary or traveling. For traveling configurations, the shapes resemble those in the last panels of Figures \ref{fig:0102} and \ref{fig:0104}. For $\beta$ values farther from $\beta_0$\,, intermediate shapes appear roughly halfway between these extremes.

Figure \ref{fig:bif1a} shows the high tension case ($\gamma=3/4$) with bifurcation point $\beta_{0} = 0.0919$. The trivial equilibrium is stable for $\beta > \beta_0$ and loses its stability for values $\beta < \beta_0$\,, consistent with Table \ref{table:bif}. For each $\beta$ we plot the final, near steady-state velocities $|\dot{\bf z}_i|$ of all 40 filaments; in this regime, all filaments share the same velocity, either zero or a constant nonzero value.
\vspace{-0.4in}

\begin{figure}[ht] \centering
\includegraphics[width=0.85\textwidth]{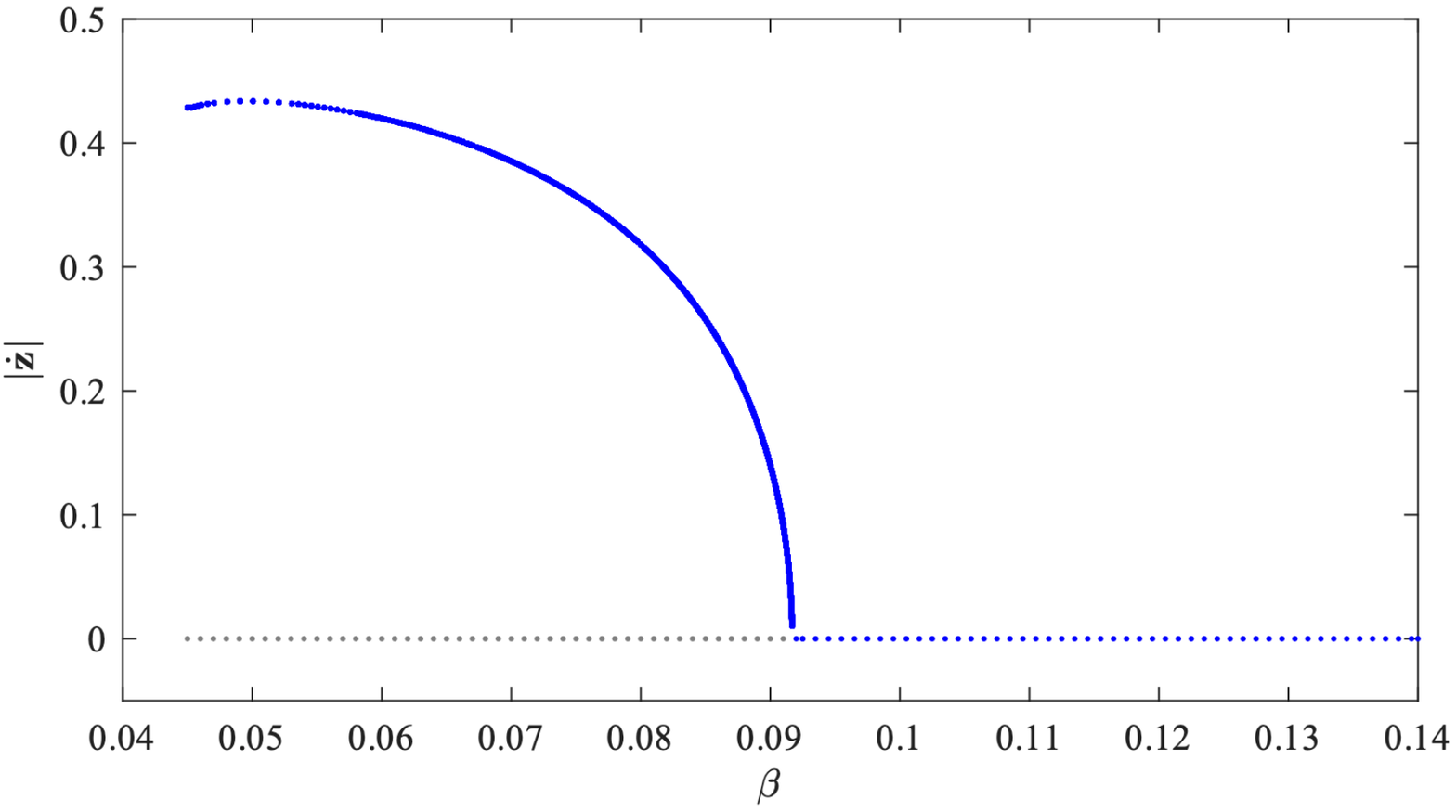}

\vspace{-0.4in}

\caption{Plot of the bifurcation behavior showing the speed of the filament ensemble vs. $\beta \in (0.045,0.14)$ at $t=10^6$ with $\gamma=3/4$ (high tension) and $\beta_0=0.0919$. Gray points correspond to unstable steady state solutions while blue correspond to stable steady state solutions.}
\label{fig:bif1a}
\end{figure}    

Figure \ref{fig:bif1b} shows the low-tension case ($\gamma=1/4$, $\beta_0 = 0.2757$) where the dynamics are more complex. As before, $|\bfz_i|$ are plotted for all $i$\,. While only a single point appearing for a given $\beta$ corresponds to translation at constant velocity with no rotation, multiple points appearing corresponds to more complex behaviors, detailed below. Blue points correspond to stable steady state solutions and gray points to unstable steady state solutions. Continuation was used to obtain solutions for $\beta \ge 0.205$ while uniform spacing was used elsewhere ($\beta = 0.135, 0.137, \ldots$). Simulations were run for each point in the uniformly sampled region using the initial conditions from Example \ref{ex:small} (near the trivial solution) on the left, and the initial condition from Example \ref{ex:bi} on the right. Red and orange correspond to maximal and minimal velocities, respectively, for each node $i$ over $t\in[900,1000]$. For $\beta > \beta_0$\,, the trivial equilibrium is stable. For $\beta < \beta_0$ solutions diverge away from the trivial steady state to a nontrivial steady state like the one shown in Example \ref{ex:small4} or to more complex behavior, seen in Figure \ref{fig:rep}. Taken together, these plots indicate the parameter ranges corresponding to stable trivial states, bifurcating steady states, and more complex dynamics.

\begin{figure}[ht] \centering
\includegraphics{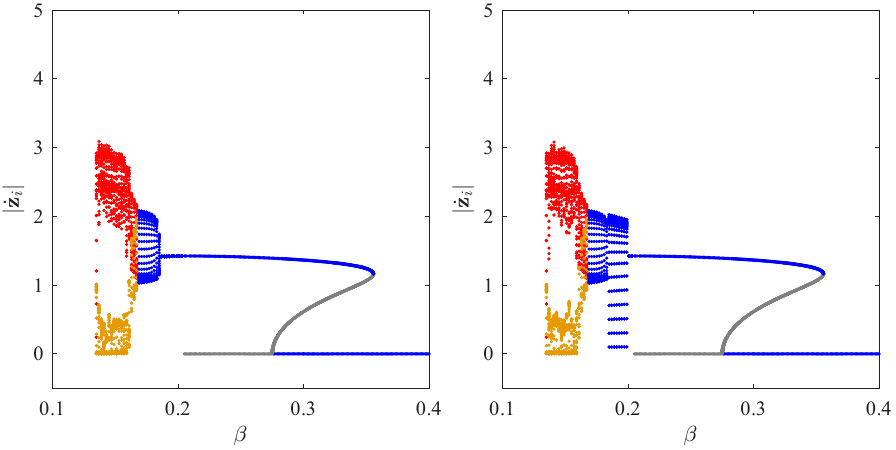}
\caption{Plot of the bifurcation behavior of solutions shown by plotting $|\dot{\bf z}_i|$ for stable steady states (blue), unstable steady states (gray), and long term behaviors for values of $\beta$ that seem to lack steady states (orange and red). Orange and red correspond to $\max_{t\in [900,1000]} |\dot{\bf z}_i|$ and $\min_{t\in [900,1000]} |\dot{\bf z}_i|$ for $i = 1,\ldots, n$\,. In the uniformly sampled region, the left plot uses the initial condition from Example \ref{ex:small} and the right plot uses the initial condition from Example \ref{ex:bi}. Here, $\beta \in (0.135,0.4)$ with $\gamma=1/4$ and $\beta_0 = 0.2757$. 
}
\label{fig:bif1b}
\end{figure} 

For $\beta_0 \lesssim 0.2$, maximal and minimal velocities frequently disagree, corresponding to three distinct behaviors reflected in Figure \ref{fig:rep}: \textbf{spinning motion} (filaments rotate about their center of mass), \textbf{whirling motion} (filaments rotate about a point far from the center of mass), and \textbf{chaotic motion} (neither periodic nor translational). These distinctions are important because they signal qualitative changes in the system: spinning and whirling reflect loss of simple translational invariance, while chaotic motion marks the onset of complex, potentially unpredictable filament interactions. In Figure \ref{fig:bif1b}, the right plot shows approximately $0.201 > \beta_\text{spinning} > 0.184 > \beta_\text{whirling} > 0.167 > \beta_\text{chaotic}$; on the left, only whirling and chaotic behaviors are observed using the initial condition from Example \ref{ex:small}. Comparing the two plots shows that long term behaviors depend on the initial condition used.

\begin{figure}[ht] \centering
\includegraphics[scale=0.85]{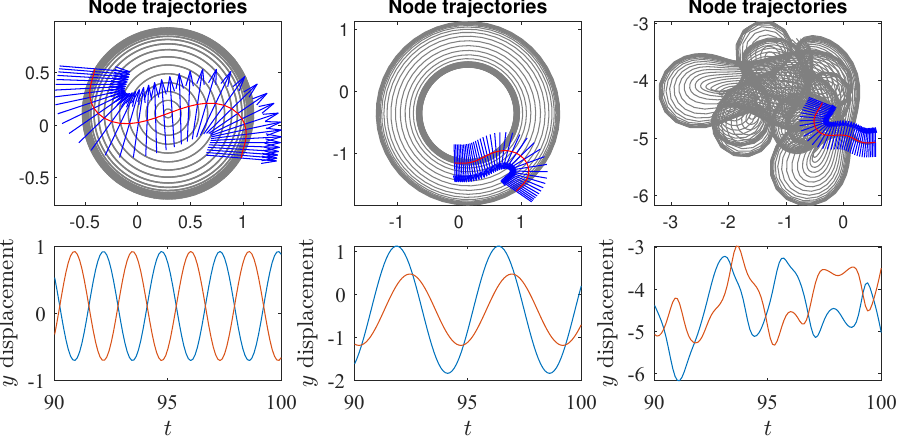}
\caption{Plot of typical nontrivial behaviors that can take place for low $\beta$ values, $\beta =$ 0.193 (left), 0.174 (middle), and 0.146 (right).  The top row shows trajectories of the centers of masses (gray) for each of the 40 discretized fibers from time 90 to 100.  It also includes fiber orientations (blue arrows) and centers of mass (red line) at time 90.  Motion for left and middle is counterclockwise while rightmost chaotic motion is sometimes clockwise and sometimes counterclockwise.  The bottom row shows the $y$ displacement for the nodes at the ends of the set of fibers.  Left and middle show decidedly periodic motion while the rightmost plot shows chaotic motion.
\label{fig:rep}}
\end{figure}  

\section{Discussion}
In this work, we have derived and analyzed a simplified continuum model for an ensemble of rigid filaments subject to friction, repulsive pressure, boundary confinement, and tension. By assuming a limit of large filament numbers, we obtained a coupled system of nonlinear partial differential equations for the filament positions and orientations.

Our central mathematical result is the identification of a stability threshold governed by the dimensionless parameter $\beta$\,. As shown in Section \ref{sec:lin}, the set of trivial steady states corresponding to a rectangular, non-moving filament ensemble is spectrally stable when $\beta > \beta_0$\,. Physically, this regime corresponds to a situation in which the stabilizing effects of tension and boundary confinement dominate over the repulsive pressure. However, as the parameter $\beta$ decreases, representing an increase in repulsive pressure or a decrease in tension, the system undergoes a symmetry-breaking instability.

Our formal bifurcation analysis and numerical simulations reveal that this instability is the normal form of a pitchfork bifurcation. However, owing to the rotational symmetry of the system, this is not a standard one-dimensional pitchfork, instead, the bifurcating manifold is two-dimensional.
The nature of this bifurcation, whether super- or sub-critical, depends on the tension parameter $\gamma$. The sub-critical case is of particular interest mathematically as it implies the existence of a hysteresis loop where the system may jump to a large-deformation branch even before the linear stability threshold is crossed. Beyond transitioning to steadily moving deformed states, our numerical investigations demonstrate that the system also exhibits more complex behaviors, including periodic oscillations and chaotic dynamics.
This suggests that in certain parameter regimes, the filament network is highly sensitive to perturbations, potentially allowing the system to switch between stable symmetric states, persistent traveling waves, and even chaotic regimes.

A key feature of bifurcated solutions is the persistence of motility. In the trivial steady state, the boundary forces on the outermost filaments point in opposite directions and cancel out, resulting in zero net force. However, in the bifurcated state, the rotational symmetry is broken. The filaments fan out, causing the boundary forces that remain perpendicular to the filament orientation to have a non-zero sum. This net force propels the entire ensemble, leading to traveling wave solutions. This provides a purely mechanical explanation for the generation of net propulsion, driven solely by the internal pressure of the filament network.

From a biological perspective, this simplified model offers insight into the role of Coulomb repulsion in the Filament Based Lamellipodium Model (FBLM) \cite{MOSS}. While often modeled merely as a volume-exclusion effect to prevent collapse, our results suggest that repulsion can act as an active drive for symmetry breaking and motility initiation. The \textit{pressure} forces the filaments to splay, and the geometric constraints translate this splaying into translational motion.

Although pressure drives this lateral splaying, our analysis shows that it also introduces a destabilizing effect in the longitudinal direction. To balance this, our model incorporates a tension force acting between the centers of mass of the filaments. It is important to note that this specific center-of-mass tension force is a mathematical abstraction and does not have a direct biological significance. However, it serves as a satisfactory substitute in this simplified setting. In the full FBLM and in actual biological cells, cell membrane tension and myosin contraction provide analogous stabilizing effects.

There are, of course, limitations to this caricature of the FBLM. We have treated filaments as rigid rods, ignoring the bending energy that is central to the full model \cite{MOSS, OelzCSSmall08}. The inclusion of filament flexibility would introduce a fourth-order derivative in the spatial variable, significantly complicating the analysis but likely providing a regularization of the large deformations seen in the sub-critical branch. Furthermore, the biological reality involves complex polymerization and depolymerization kinetics which we have neglected here to focus on the mechanical stability \cite{pollard2007regulation}. Future work should address the inclusion of bending elasticity to determine if it suppresses the sub-critical nature of the bifurcation. Additionally, extending this stability analysis to the full 2D setting of the FBLM, rather than the 1D cross-section considered here, remains a challenging but necessary step toward a complete understanding of actin network stability.

\section*{Statements and Declarations}

\subsection*{Conflict of Interest}
The authors declared that they have no conflict of interest.

\subsection*{Acknowledgments} GMA acknowledges support from the Ernst Mach Grant (ASEA-UNINET) via Austria's Agency for Education and Internationalisation (OeAD), funded by the Federal Ministry of Education, Science and Research (BMBWF), and from the Natural Sciences Research Institute (NSRI), University of the Philippines Diliman, project code MAT-24-1-01. JB is supported by NSF DMS (NIGMS) 2347956-7, NSF DMS-REU 2150108, and the IUPUI Research Support Funds Grant. CS is supported by the Austrian Science Fund (grant no. F65).

\bibliographystyle{siam}
\bibliography{pressure.bib}

\end{document}